\documentclass[11pt]{amsart}
\usepackage[utf8]{inputenc}
\usepackage{fontenc}
\usepackage{amsfonts}
\usepackage{amssymb}
\usepackage{amsmath}
\usepackage{amsthm}
\usepackage{enumerate}
\usepackage{hyperref}
\usepackage{mathrsfs}
\usepackage{tikz}
\newcommand{\mathscripty}{\mathscr}

\newcommand{\rs}{\mathord{\upharpoonright}}

\newcommand{\forces}{\Vdash}

\newcommand{\NN}{\mathbb{N}}

\newcommand{\ce}{\mathbb{C}}

\newcommand{\SB}{\mathscripty{B}}

\newcommand{\SK}{\mathscripty{K}}

\newcommand{\SM}{\mathscripty{M}}

\newcommand{\SQ}{\mathscripty{Q}}

\newcommand{\Cstar}{\mathrm{C}^*}
\newcommand{\cst}{\mathrm{C}^*}
\newcommand{\cstar}{$\mathrm{C}^*$}
\newcommand{\cC}{\mathcal{C}}
\newcommand{\cU}{\mathcal{U}}
\newcommand{\cZ}{\mathcal{Z}}

\newcommand{\cP}{\mathcal{P}}
\newcommand{\bbN}{\mathbb{N}}

\newcommand{\bbO}{\mathbb{O}}
\newcommand{\bbR}{\mathbb{R}}

\newcommand{\cO}{\mathcal{O}}
\newcommand{\BBB}{D}

\newcommand{\ZFC}{\mathrm{ZFC}}
\newcommand{\bbP}{\mathbb P}
\newcommand{\bbC}{\mathbb C}
\newcommand{\bbA}{\mathbb A}

\newcommand{\Q}{\SQ}

\newtheorem{theorem}{Theorem}[section]
\newtheorem*{theorem*}{Theorem}
\newtheorem{proposition}[theorem]{Proposition}
\newtheorem*{proposition*}{Proposition}
\newtheorem{lemma}[theorem]{Lemma}
\newtheorem*{lemma*}{Lemma}
\newtheorem{corollary}[theorem]{Corollary}
\newtheorem*{corollary*}{Corollar}

\newtheorem*{fact*}{Fact}
\theoremstyle{definition}
\newtheorem{definition}[theorem]{Definition}
\newtheorem*{definition*}{Definition}
\newtheorem{claim}[theorem]{Claim}
\newtheorem*{claim*}{Claim}

\newtheorem*{conjecture*}{Conjecture}

\newtheorem{theoremi}{Theorem}

\theoremstyle{remark}

\newtheorem*{example*}{Example}
\newtheorem{remark}[theorem]{Remark}
\newtheorem*{remark*}{Remark}

\newtheorem*{note*}{Note}
\newtheorem{question}[theorem]{Question}
\newtheorem*{question*}{Question}

\DeclareMathOperator{\Fin}{Fin}

\DeclareMathOperator{\Ad}{Ad}

\DeclareMathOperator{\Ext}{Ext}
\DeclareMathOperator{\Extw}{Ext^w}
\DeclareMathOperator{\Extcpc}{Ext^w_{cpc}}

\DeclareMathOperator{\OS}{OS}
\DeclareMathOperator{\cb}{cb}
\DeclareMathOperator{\SOP}{SOP}

\newcommand{\bfT}{{\mathbf T}}
\newcounter{my_enumerate_counter}
\newcommand{\pushcounter}{\setcounter{my_enumerate_counter}{\value{enumi}}}
\newcommand{\popcounter}{\setcounter{enumi}{\value{my_enumerate_counter}}}

\usepackage{enumitem}

\begin{document}

\title{The Calkin algebra is $\aleph_1$-universal}%

\author[Ilijas Farah]{Ilijas Farah}
\address[I. Farah]{Department of Mathematics and Statistics,
York University,
4700 Keele Street,
North York, Ontario, Canada, M3J
1P3}
\email{ifarah@mathstat.yorku.ca}
\urladdr{http://www.math.yorku.ca/$\sim$ifarah}
\author{Ilan Hirshberg} 

\address[Ilan Hirshberg]{Department of Mathematics\\
 Ben Gurion University of the Negev\\
  P.O.B. 653, Be'er\\
Sheva 84105, Israel}
\email{ilan@math.bgu.ac.il}

\urladdr{http://www.math.bgu.ac.il/~ilan/}

\author[Alessandro Vignati]{Alessandro Vignati}
\address[A. Vignati]{Institut de Math\'ematiques de Jussieu - Paris Rive Gauche (IMJ-PRG)\\
UP7D - Campus des Grands Moulins\\
B\^atiment Sophie Germain\\
8 Place Aur\'elie Nemours\\ Paris, 75013, France.
Currently at: Department of Mathematics, KU Leuven, Celestijnenlaan 200B, B-3001 Leuven, Belgium}

\email{ale.vignati@gmail.com}
\urladdr{http://www.automorph.net/avignati}


\thanks{IH and AV's visit to Toronto were supported by  NSERC. IH was supported by the Israel Science Foundation, grant no.~476/16. IF's visit to CRM was  supported by the Clay Mathematics Institute. 
AV is supported by a PRESTIGE co-fund Scholarship and an FWO scholarship.}

\date{\today}%
% ----------------------------------------------------------------
\begin{abstract}
We discuss the existence of (injectively)  universal \cstar-algebras and 
prove  that all $\Cstar$-algebras of density character $\aleph_1$ embed into the Calkin algebra, $\SQ(H)$.
Together with other results, this shows that each of the following assertions 
is relatively consistent with $\ZFC$: (i) $\SQ(H)$ is a 
$2^{\aleph_0}$-universal \cstar-algebra.   (ii) There exists a $2^{\aleph_0}$-universal \cstar-algebra, 
but  $\SQ(H)$ is not $2^{\aleph_0}$-universal. (iii) A $2^{\aleph_0}$-universal  \cstar-algebra does not exist. 
We also prove that it is relatively consistent with $\ZFC$  that (iv) there is no $\aleph_1$-universal nuclear \cstar-algebra, 
and that (v) there is no  $\aleph_1$-universal simple nuclear \cstar-algebra. 
  \end{abstract}
 
\maketitle

\section{Introduction}

Let $H$ denote  the separable infinite-dimensional complex Hilbert space. The Calkin algebra $\SQ(H)$ is the quotient $\SB(H)/\SK(H)$ of the algebra $\SB(H)$ of all bounded linear operators on $H$ over the ideal of all compact operators. 

Given a category $\cC$ of metric structures and a cardinal $\kappa$, an object $A\in \cC$ is  \emph{(injectively)\footnote{The dual notion, surjective universality, is trivialized in the category of unital \cstar-algebras. Since every unital \cstar-algebra is generated by its unitary group,  the full group \cstar-algebra associated with the free group $F_\kappa$, $\cst(F_\kappa)$, is surjectively $\kappa$-universal for every infinite cardinal $\kappa$, and, in the abelian setting, $C([0,1]^\kappa)$ is surjectively $\kappa$-universal.} $\kappa$-universal} if it has density character $\kappa$\footnote{The density character of a metric space is the smallest cardinality of a dense subset.} and every object $B\in \cC$ of density character at most $\kappa$ is isometric to a substructure of $A$. 

%Let $A$ be a $\Cstar$-algebra and let $\mathcal C_A$ be the class of all algebras embedding in some ultrapower of $A$. We say that $A$ is $\aleph_1$-universal if all algebras of density character
% $\aleph_1$ belonging to $\mathcal C_A$ embed into $A$. If $A$ is unital, then all subalgebras and embedding considered are required to be so. In case $A=\ell_\infty/c_0$, then $\mathcal C_A$ is the class of unital abelian $\Cstar$-algebras, while if $A=C(X)$ for some infinite connected compact $X$, then $\mathcal C_A$ is the class of unital abelian projectionless $\Cstar$-algebras. Since every separable $\Cstar$-algebra embeds into $\mathcal C(H)$, $\mathcal C_{\mathcal C(H}$ is the class of all $\Cstar$-algebras. That $\ell_\infty/c_0$ is $\aleph_1$-universal is Parovi\v{c}enko's Theorem (\cite{Parovicenko}), while $\aleph_1$-universality of $C(\beta\er\setminus\er)$ was proved in \cite{DowHart.Universal}. Whether $\mathcal C(H)$ is $\aleph_1$-universal was asked specifically by the second author in \cite[Question E]{V.PhDThesis}. The following is our main result:

The question whether the Calkin algebra can be $\aleph_1$-universal for the category of $\Cstar$-algebras 
answered in Theorem~\ref{T:mainintro} 
was raised by Piotr Koszmider (personal correspondence) and in  \cite[Question E]{V.PhDThesis}. 
%The following answers it.

\begin{theoremi}\label{T:mainintro}
All $\Cstar$-algebras of density character at most  $\aleph_1$ embed into the Calkin algebra. Therefore the Continuum Hypothesis implies that the Calkin algebra is an $\aleph_1$-universal \cstar-algebra.  
\end{theoremi}

One of the ingredients of our proof is the analysis of the $\Extw$-group of simple, separable, and unital \cstar-algebras that tensorially absorb the Cuntz algebra~$\mathcal O_2$.   

By a fundamental result of Kirchberg (\cite{KirchPhil}), 
 $\cO_2$ is the universal separable nuclear \cstar-algebra, and even the universal separable exact \cstar-algebra. 
 %It is natural to ask whether $\aleph_1$-universal nuclear (or nuclear simple) \cstar-algebras could exist. 
 The following theorem (and more; see Theorem~\ref{T.nuclear+}) will be proved in \S\ref{S.nuclear}.  
 
\begin{theoremi}\label{T:nuclear}
It is relatively consistent with $\ZFC$  that there is no $\aleph_1$-universal nuclear
\cstar-algebra, and that there is no $\aleph_1$-universal nuclear, simple, \cstar-algebra. 
\end{theoremi} 

Our proof of Theorem~\ref{T:nuclear} requires basic command of the method of forcing, 
as exposed e.g., in \cite[IV]{Ku:Set}. 
 
%It is not known whether every exact \cstar-algebra is isomorphic to a subalgebra of a nuclear \cstar-algebra. 

\subsection*{Acknowledgments} 
We would like to thank George Elliott, Jamie Gabe, Piotr Kosz\-mi\-der,  N. Christopher Phillips, and Chris Schafhauser for helpful remarks and to Bradd Hart for pointing out 
to some omissions in the early draft of the present paper. 
We would also like to thank Bartosz Kwasniewski for bringing \cite{kwasniewski2016aperiodicity}, used in the proof of Lemma~\ref{L.nuclear}, 
to our attention. 
Part of this work was completed during AV's and IH's visits to Toronto in the summer of 2017 and   
the  authors' visit to the program ``IRP Operator Algebras: Dynamics and Interactions'' at CRM (Barcelona). 
The authors would like to thank the organizers, in particular Francesc Perera,  for their support  and  hospitality.

\section{The proof of Theorem~\ref{T:mainintro}}\label{s:theproof}

This section is entirely devoted to the proof of Theorem~\ref{T:mainintro}. Familiarity with model theory, in particular  axiomatizability and the different layers of saturation,  is required (see \cite{Muenster}, or \cite{FaHa:Countable} for an overview of the concept of saturation). For  information on   \cstar-algebras see~\cite{Blackadar.OA} and for analytic K-homology see \cite{HigsonRoe}.

The main technical difficulty in the proof of Theorem~\ref{T:mainintro} is posed by the absence of reasonable saturation properties in the Calkin algebra. The simplest instance of this is  the fact that  the image of the unilateral shift  in~$\SQ(H)$ is a unitary with full spectrum and no square root. As pointed out in the introduction to \cite{PhWe:Calkin}, this implies that $\SQ(H)$ is not injective (in a categorical sense) for separable \cstar-algebras and complicates  construction of outer automorphisms of $\SQ(H)$.  More sophisticated obstructions to saturation, and even homogeneity, 
in $\SQ(H)$ were exhibited in  \cite[\S 4]{FaHa:Countable} and  \cite{FH.CalkinHom}, respectively. All of these obstructions are
of K-theoretic nature.  

 A unital \cstar-algebra $A$ is purely infinite and simple if it is infinite dimensional and for every nonzero positive $a\in A$ there is $x\in A$ such that $x a x^*=1$.  The Cuntz algebra $\mathcal O_2$ is the universal \cstar-algebra generated by two isometries $s$ and $t$   satisfying
  \[
  s^*s=t^*t=1=ss^*+tt^*. 
  \]
Let 
\[
\bbO=\{A: \text{$A$  is unital, purely infinite, simple, and $A\otimes\mathcal O_2\cong A$}\}.
\] 
(Since $\mathcal O_2$ is nuclear,  there is no ambiguity in what tensor product is used. In this case there is a unique $\Cstar$-norm on the algebraic tensor product.)

\begin{lemma} 
\label{L:embed}
Every \cstar-algebra $A$ embeds into a \cstar-algebra $B\in \bbO$ of the same density character as $A$.  If $A$ is unital then the embedding can be chosen to be unital. 
\end{lemma} 

We provide two proofs of Lemma~\ref{L:embed}, one of model-theoretic and one of 
operator-algebraic flavour. 

\begin{proof}[The first proof of Lemma~\ref{L:embed}]
The class $\bbO$ is separably axiomatizable by  (\cite[Theorem~2.5.1 and Theorem~2.5.2]{Muenster}). We first consider the case when $A$ is separable. Then $A$ is isomorphic to a subalgebra of $\SQ(H)$, and the embedding can be chosen to be unital if $A$ is unital.  By the downward L\"owenheim--Skolem theorem (\cite[Theorem~2.6.2]{Muenster}) we can find a separable elementary submodel $C$ of  $\SQ(H)$ into which $A$ embeds. Then $C$ is simple and purely infinite, and $C\otimes \mathcal O_2$ is as required. 

Now suppose $A$ is not separable and let $\kappa$ be its density character.  Again by the downward L\"owenheim--Skolem theorem we can find a separable elementary submodel $A_0$ of $A$.  By the first paragraph, we can find a separable $B_0\in \bbO$ into which $A_0$ embeds.  By the standard elementary chain argument (\cite[Proposition 7.10]{BYBHU}) we construct a $\kappa$-saturated elementary extension $B_1$ of $B_0$. Writing $A$ as a union of an elementary chain of submodels of density character $<\kappa$ and using the saturation of $B_1$, we can embed $A$ into $B_1$. Again by the downward L\"owenheim--Skolem theorem we can find $B_2$ of density character $\kappa$ such that $A\subseteq B_2\subseteq B_1$ and $B_2$ is elementary equivalent to $B_1$, and therefore to $B_0$. Being elementarily equivalent to $B_0$, $B_2$ is purely infinite and simple,~$B_2$ might not be $\mathcal O_2$-stable, essentially  by \cite{Gha:SAW*}, however $B=B_2\otimes \mathcal O_2$ satisfies all requirements. 
\end{proof}

The following proof relies on a result proved in \S\ref{S:more}. 

\begin{proof}[The second proof of Lemma~\ref{L:embed}]
The Cuntz-Pimsner algebra $\mathcal O_E$ (see Lemma~\ref{L.nuclear}), constructed in \cite[\S 4]{BrownOzawa.FinDim}, is simple, and has the same density character as $A$. Moreover $A$ embeds unitally into it. $\mathcal O_E\otimes\mathcal O_2$ then provides the necessary object.
\end{proof}

%An alternative (and more natural from the model-theoretic point of view)  approach to the proof of Lemma~\ref{L:embed} uses the axiomatizable class $\bbO':=\{A: A$  is  purely infinite, simple, and potentially $\cO_2$-absorbing$\}$ (see \cite{FaHaRoTi:Relative})  in place of $\bbO$. 

%An alternative (and more natural from the model-theoretic point of view)  approach to the proof of Lemma~\ref{L:embed} uses that every separable unital \cstar-algebra embeds unitally into the Calkin algebra, and that therefore $A$ embeds into an ultrapower of $\SQ(H)$. The latter purely infinite and simple, and since being purely infinite and simple is axiomatizable, it is possible to construct $M$, an elementary submodel of $\SQ(H)$ with the same density character of $A$, in which $A$ embeds. $M\otimes \mathcal O_2$ would provide the necessary object for the proof of Lemma~\ref{L:embed}.

We shall need the semigroups $\Extw(A)$ and  $\Extcpc(A)$ associated to a separable and unital \cstar-algebra $A$. An  injective unital $^*$-homomorphism $\pi \colon A \to \SQ(H)$ is the Busby invariant of an extension of $A$ by  $\SK(H)$. By a slight abuse of terminology, we say that such a  $^*$-homomorphism is an \emph{extension} (see \cite[Proposition~2.6.3]{HigsonRoe}). Two extensions $\theta_j\colon A\to \SQ(H)$, for $j=1,2$ are  \emph{weakly equivalent} if there is a unitary $u\in \SQ(H)$ such that $\theta_1=\Ad u\circ \theta_2$. 

Since $\SM_2(\Q(H)) \cong \SQ(H)$,  the set of extensions of $A$ is equipped with the direct sum operation.  The set of weak equivalence classes of extensions of $A$  forms a semigroup, denoted  $\Extw(A)$.  An extension $\theta\colon A\to \SQ(H)$ is \emph{semisplit} if there exists a completely positive contraction (c.p.c.) $\varphi\colon A\to \SB(H)$ such that (denoting the quotient map from $\SB(H)$ onto $\SQ(H)$ by $\pi$)  $\pi\circ\varphi=\theta$. If $\varphi$ is a unital $^*$-homomorphism then we say that $\theta$ is \emph{split}.  A split extension exists when $A$ is separable, and  Voiculescu's theorem (\cite[Theorem~3.4.7]{HigsonRoe})  implies that it acts as the unit in $\Extw(A)$. 
Let
 \[
 \Extcpc(A):=\{\theta\in \Extw(A): \theta\text{ is semisplit}\}. 
 \]
 Stinespring's theorem (\cite[Theorem~II.6.9.7]{Blackadar.OA})  easily implies that $\Extcpc(A)=\Extw(A)^{-1}$, the group of all invertible elements of $\Extw(A)$. 

The following is a standard application of  quasicentral approximate units. 

\begin{lemma}\label{lemma:lift}
Suppose that a separable \cstar-algebra $A$ is an inductive limit of subalgebras $A_n$, for $n\in \bbN$. If  $\theta\colon A\to \SQ(H)$ is an extension such that its restriction to $A_n$ is semisplit for every $n$, then $\theta$ is semisplit.  
\end{lemma}

\begin{proof} 
Let $\delta_n>0$ be small enough so that for all operators $e$ and $a$ satisfying $0\leq e\leq 1$,  $\|a\|\leq 1$,  and $\|[e,a]\|<\delta_n$ we have $\|[e^{1/2}, a]\|<2^{-n}$. Let $\psi_n$ be a c.p.c.  lift for $\theta\restriction A_n$.  By the Arveson Extension Theorem (\cite[Theorem II.6.9.12]{Blackadar.OA}) we can extend $\psi_n$ to a c.p.c. map  $\tilde\psi_n\colon A\to \SB(H)$. Let $E=\pi^{-1}(\theta(A))$ and let $a_n$, for $n\in \bbN$, be an enumeration of a dense subset of the unit ball of $A$ whose intersection with the unit ball of $A_n$ is dense for all $n$.    By  \cite[Proposition 3.2.8]{HigsonRoe} we can find a sequence $f_n$, for $n\in \NN$, which is an approximate identity for $\SK(H)$  that is quasicentral for $E$.     By refining this sequence, we may assume that the following conditions hold for all $i,j,k$, and $n$  with $i,j,k\leq n$. 
   \begin{enumerate}
\item    $\|[f_n, \tilde\psi_i(a_j)]\|<\delta_n$. 
   \item $\|(1-f_n)(\psi_i(a_j)-\psi_k(a_j))\|<2^{-n}$, if $a_j\in A_i\cap A_k$.  
   \end{enumerate}
The  second condition can be assured because the assumptions imply $\psi_i(a_j)-\psi_k(a_j)$ is compact,  and the first condition can be assured by the quasicentrality of the sequence. 

   Given these conditions we have  $\|[(f_{n+1}-f_n)^{1/2},\tilde\psi_n(a_j)]\|<2^{-n}$   for all $j\leq n$. Therefore 
\[
\psi(a)=
\sum_n (f_{n+1}-f_n)^{1/2}\tilde\psi_n(a)(f_{n+1}-f_n)^{1/2}
\]
is well-defined since the finite partial sums converge in the strong operator topology. Since every partial sum is c.p.c., so is $\psi$.  For all $a_j\in A_i$ and all $n$ we 
also have that $\psi(a_j)-\psi_n(a_j)$ is compact and therefore $\psi$ is a c.p.c. lift of $\theta$ as required. 
\end{proof}

\begin{proposition} \label{P.Ext} 
Suppose $A\in \bbO$ is separable. Then $\Extcpc(A)=0$
\end{proposition} 
\begin{proof}
Recall that an endomorphism $\varphi$ of a \cstar-algebra $A$ is \emph{asymptotically inner}  if there exists a continuous path of unitaries $u_t$, for $0\leq t<\infty$, such that $u_0=1$ and $\varphi(a)=\lim_{t\to \infty} (\Ad u_t) a$ for all $a\in A$. It follows from  \cite[Lemma 2.2.1]{Phil:Class} that any unital endomorphism of $\mathcal{O}_2$ is asymptotically inner. 
 
Suppose $A\cong A\otimes \mathcal O_2$ and let $s$ and $t$ be two standard generators of~$\mathcal O_2$.  It follows that the endomorphism  $\zeta(a)= (1\otimes s)a(1\otimes s^*)+(1\otimes t)a(1\otimes t^*)$ of $A$ is asymptotically inner. 

Since $\Extcpc(A)$ is a group, it suffices to  prove that each of its  elements is idempotent. 
The semigroup of endomorphisms of $A$ acts on $\Extcpc(A)$ by composition: if $\zeta\colon A\to A$ and $\theta\colon A\to \SQ(A)$ is an extension, then $\zeta.\theta=\theta\circ \zeta$ is an extension of $A$. However, as elements of $\Extcpc(A)$, we have $[\theta\circ \zeta] = [\theta] + [\theta]$. By \cite[Corollary 18.5.4]{blackadar.KT}, $\Extcpc(A)$ is homotopy invariant.\footnote{To match the notation in Blackadar's book, what we denote by $\Extw(A)$ is denoted there by $\Ext^u_w(A,\mathbb{C})$. By \cite[Proposition 15.14.2]{blackadar.KT}, for any unital separable \cstar-algebra $A$ we have $\Ext^u_w(A) \cong \Ext(A,\mathbb{C})$. Thus $\Extcpc(A) \cong \Ext(A,\ce)^{-1}$.}
  Thus $\Extcpc(A)$ is trivial, as required.
\end{proof} 

We are now ready to prove Theorem~\ref{T:mainintro}. By Lemma~\ref{L:embed} it suffices to prove that every $A\in \bbO$ of density character $\aleph_1$ embeds into $\SQ(H)$. By the downward L\"owenheim--Skolem theorem there exists an increasing chain $A_\alpha$, for $\alpha<\aleph_1$, of separable elementary submodels  satisfying $A=\bigcup_{\alpha<\aleph_1} A_\alpha$ and for every limit ordinal $\delta<\aleph_1$ we have $A_\delta=\overline{\bigcup_{\alpha<\delta} A_\alpha}$. Each  $A_\alpha$ is unital, purely infinite and simple, and it absorbs $\mathcal O_2$,  since these properties are  elementary for separable \cstar-algebras   (\cite[Theorem~2.5.1 and Theorem~2.5.2]{Muenster}).

We want to find extensions $\varphi_\alpha\in \Extcpc(A_\alpha)$ such  for all $\alpha<\beta<\aleph_1$ we have 
\[
\varphi_\alpha\in \Extcpc(A_\alpha)\text{ and } 
\varphi_\beta\rs A_\alpha=\varphi_\alpha. 
\]
Choose  $\varphi_0\in\Extcpc(A_0)$. Suppose $\varphi_\alpha$ has been defined for all $\alpha<\beta$. 

 If $\beta$ is a successor ordinal, let $\alpha$ be such that $\alpha+1=\beta$.   Fix  $\psi\in \Extcpc(A_\beta)$.  Then Proposition~\ref{P.Ext} implies that both $\psi':=\psi\rs A_\alpha$ and $\varphi_\alpha$ are split. By Voiculescu's theorem (\cite[Theorem~3.4.7]{HigsonRoe})  there exists a unitary $u\in \SQ(H)$ such that $\varphi_\alpha=\Ad u \circ \psi'$ and $\varphi_\beta=\Ad u\circ \psi$ is as required.

Now suppose $\beta$ is a limit ordinal. Then $\varphi_\beta$ is already defined on a dense subalgebra $\bigcup_{\alpha<\beta} A_\alpha$ of $A_\beta$, and  Lemma~\ref{lemma:lift} implies that its continuous extension to $A$  is semisplit.

This describes the recursive construction. Since $A=\bigcup_{\alpha<\aleph_1} A_\alpha$, for all  $a\in A$ the ordinal $\alpha(a)=\min\{\alpha: a\in A_\alpha\}$ is well-defined. Then $\Phi(a)=\varphi_{\alpha(a)} (a)$ extends each $\varphi_\alpha$, provides the desired embedding of $A$ into $\SQ(H)$,  and completes the proof of Theorem~\ref{T:mainintro}. 

%\begin{corollary}
%Let $A$ be a unital separable $\Cstar$-algebra and $B=A\otimes \SK(H)$. Then the corona algebra of $B$ is universal.
%\end{corollary}
%\begin{proof}
%The algebra $C=1\otimes\SK(H)$ embeds in $B$. Take a strictly positive element $a\in C$, and note that $a$ is a strictly positive element for $B$. By realizing the multiplier algebra as the idealizer of $B$, we can find a unital embedding of $\SM(C)\cong\SB(H)$ into $\mathcal M(B)$. This embedding induces a unital embedding of the Calkin algebra in $\SM(B)/B$.
%\end{proof}

\section{Corollaries of Theorem~\ref{T:mainintro} and related results} 

By GCH, we mean the \emph{Generalized Continuum Hypothesis}, which is the statement $2^{\lambda}=\lambda^+$ for every cardinal $\lambda$.
\begin{definition} 
Following 
\cite{kojman1992nonexistence} we say that 
a cardinal $\kappa$ is  \emph{far from the GCH} there exists a cardinal $\lambda$ such that the following holds
($\lambda^+$ denotes the least cardinal greater than $\lambda$):
\begin{equation}\label{E:kappa-lambda}
 \lambda^+<\kappa<2^\lambda. \tag{$\star$}
 \end{equation} 
\end{definition} 

\begin{corollary} \label{C1} 
Each of the following assertions is relatively consistent with $\ZFC$ : 
\begin{enumerate}
\item  \label{I1} $\SQ(H)$ is a $2^{\aleph_0}$-universal \cstar-algebra.   
\item  \label{I2} There exists a $2^{\aleph_0}$-universal \cstar-algebra, but  $\SQ(H)$ is not $2^{\aleph_0}$-universal. 
\item \label{I3} A $2^{\aleph_0}$-universal  \cstar-algebra does not exist. 
\end{enumerate}
\end{corollary} 

\begin{proof} \eqref{I1} 
Assume the Continuum Hypothesis.  Theorem~\ref{T:mainintro} implies that $\SQ(H)$ is $2^{\aleph_0}$-universal. 

\eqref{I2} We shall prove that the   Proper Forcing Axiom, PFA, implies the conclusion.\footnote{Readers concerned with the consistency strength issues may rest assured that only a small fragment of PFA with zero large cardinal strength is required in \cite{V.PhDThesis}.} If $2^\kappa=2^{\aleph_0}$ for all $\kappa<2^{\aleph_0}$, then  \cite[Proposition 7.10]{BYBHU} implies that the theory of $\SQ(H)$ has a saturated model $B$ of density character $2^{\aleph_0}$. By Lemma~\ref{L:embed} (and its proof), such $B$ is a $2^{\aleph_0}$-universal \cstar-algebra. It is well-known that PFA implies $2^{\aleph_1}=2^{\aleph_0}=\aleph_2$ (e.g. see \cite{Ku:Set}).  

 By \cite[Corollary 5.3.14 and Theorem 5.3.15]{V.PhDThesis} (also \cite{MKAV.AC}) PFA implies 
that there exist a closed subset $X$ of $\beta\bbN\setminus \bbN$ such that $C(X)$ does not embed into the Calkin algebra. Such an $X$ can be chosen as follows. Let $\cZ_0=\{S\subseteq \bbN: \lim_{n\to \infty} |S\cap n|/n=0\}$, the ideal of asymptotic density zero sets. (Any other dense analytic P-ideal would do in place of $\cZ_0$; see \cite{V.PhDThesis}.) Identifying $\beta\bbN$ with the set of all ultrafilters on $\bbN$, we may let $X=\{\cU\in \beta\bbN: \cU\cap \cZ_0=\{\emptyset\}\}$. 
 
\eqref{I3} By standard forcing techniques (\cite{Ku:Set}) the assertion  $\aleph_2<2^{\aleph_0}<2^{\aleph_1}$  is relatively consistent with $\ZFC$ . (For example, start from a model of GCH, add $\aleph_4$ Cohen subsets of $\aleph_1$, then add $\aleph_3$ Cohen reals.) Therefore $\kappa=2^{\aleph_0}$ and $\lambda=\aleph_1$ satisfy the  inequality \eqref{E:kappa-lambda}
and we conclude that $2^{\aleph_0}$ is far from  the GCH in this model.  

We shall prove that \eqref{E:kappa-lambda} and $\kappa^{\aleph_0}=\kappa$ together  imply that there is
  no $\kappa$-universal \cstar-algebra.\footnote{See Remark~\ref{R.I3} for   a sketch of a self-contained proof}   By \cite[\S 3.3]{EffrosRuan:OS} every Banach space embeds isometrically into an operator space (and therefore into a $\Cstar$-algebra) of the same density character. Hence  a $\kappa$-universal \cstar-algebra would also be a $\kappa$-universal Banach space. But,  by \cite[Corollary~2.4]{shelah2006banach}, \eqref{E:kappa-lambda} implies that there is no   (isometrically) $\kappa$-universal Banach space, and this concludes the proof. 
\end{proof} 

 A sketch of an alternative, self-contained (and, we believe,  more informative) proof that the condition \eqref{E:kappa-lambda} in Corollary~\ref{C1}  \eqref{I3} together with $\kappa^{\aleph_0}=\kappa$ implies there is no $\kappa$-universal \cstar-algebra is in order. It uses the  following definition  adapted to the continuous context 
  from \cite[Definition~5.1]{kojman1992nonexistence} (see   \cite[Definition~1.1]{shelah2006banach}).  

\begin{definition}\label{d:SOP}
A theory $T$ has the  Strict Order Property (SOP) if there exists a formula $\psi(\bar x,\bar y)$ of the language of $T$ in $2n$ variables for some $n\geq 1$ 
with the following two properties.  First, in every model of $T$ the relation $a\prec_\phi b$ given by 
\[
a\prec_\phi b\iff\phi(a,b)=0\text{ and } \phi(b,a)=1
\]
defines a partial ordering. Second, in every model of $T$ there are arbitrarily long finite $\prec_\phi$-chains.  
\end{definition}
 
 If the theory $T$ is complete, then by a compactness argument (i.e., taking an ultraproduct) 
 the second requirement can be replaced by the requirement that there exists an infinite $\prec_\phi$ chain in some model of $T$. 
 Since the theory of \cstar-algebras is not complete, we opt for the current formulation. 
 
 \begin{remark} 
 \label{R.I3} 
Here is the promised  alternative proof of Corollary~\ref{C1} \eqref{I3}.  
 Instead of using \cite[Corollary~2.4]{shelah2006banach}, we follow the lines of its proof. By \cite[Theorem~3.10]{kojman1992nonexistence},  \eqref{E:kappa-lambda} implies  that there is no $\kappa$-universal linear order, and moreover that any theory $T$ with SOP does not have a $2^{\aleph_0}$-universal model. As in \cite[Lemma~5.3]{FaHaSh:Model1}, consider the following condition in the language of \cstar-algebras 
\[
\varphi(x,y)=\max(|1-\|x^*x\||,|1-\|y^*y\||,  \|x^*xy^*y-y^*y\|). 
\]
Two elements $x$ and $y$ of  a \cstar-algebra $A$ satisfy $\varphi(x,y)=0$ if and only if $\|x\|=\|y\|=1$ and in the second dual  $A^{**}$ of $A$ 
the support projection of $y^*y$ is below the spectral projection of $x^*x$ corresponding to $1$. Therefore $\varphi(x,y)$ defines a partial order, $\preceq_\varphi$, on $A$. Every infinite-dimensional \cstar-algebra contains an infinite $\preceq_\varphi$ chain (consider any,  necessarily infinite-dimensional, masa or see   \cite[Lemma~5.3]{FaHaSh:Model1}).
Thus every infinite-dimensional \cstar-algebra has the 
Strict Order Property. 
 Since $\varphi$ is quantifier-free, an embedding of $A$ into $B$ is an embedding of  the poset $(A,\preceq_\varphi)$ into $(B,\preceq_\varphi)$. Hence if $C$ is a $\kappa$-universal \cstar-algebra for some cardinal $\kappa$, then every linear ordering of cardinality $\kappa$ embeds into 
the linearization of $(C,\preceq_\varphi)$, which has size $\kappa^{\aleph_0}=\kappa$. The latter is a $\kappa$-universal linear ordering, a contradiction.  
 \end{remark} 

The following is a poor man's version of Theorem~\ref{T:nuclear}. 

\begin{corollary} 
If $2^{\aleph_0}$ is far from the GCH then no \cstar-algebra of density character $2^{\aleph_0}$ is universal for all abelian \cstar-algebras of density character~$2^{\aleph_0}$.  
In particular no $2^{\aleph_0}$-universal exact \cstar-algebra. 
   \end{corollary} 
   
\begin{proof} By Remark~\ref{R.I3} the theory of $C([0,1])$ has the Strict Order Property
witnessed by a quantifier-free formula and the first claim 
follows by the argument of the latter part of Remark~\ref{R.I3}. 
Since every abelian \cstar-algebra is exact (and being abelian 
is axiomatizable, \cite[Theorem~2.5.1]{Muenster}), the second claim follows. 
\end{proof} 

The  following lemma is a special case of \cite[Proposition~2.53]{raeburn1998morita}.

\begin{lemma}\label{L.MB} 
Suppose $A$ is a \cstar-algebra and $B$ is a \cstar-subalgebra of $A$ 
that contains an approximate unit for $A$. 
Then the inclusion from $B$ into $A$ extends to an injection from $\SM(B)$ into $\SM(A)$, 
and $\SM(B)/B$ is isomorphic to a subalgebra of $\SM(A)/A$. 
\end{lemma} 
%
%\begin{proof} 
%We can identify $A$ with a nondegenerate subalgebra of $\SB(H)$ and  
% $\SM(A)$ with  the idealizer of $A$ in $\SB(H)$, 
%$\{c\in \SB(H): cA\subseteq A$ and $Ac\subseteq A\}$ (\cite[II.7.3.5]{Blackadar.OA}).
%Since $B$ has an approximate unit for $A$ it is also nondegenerate in $\SB(H)$ and 
%$\SM(B)$ can be identified with   the idealizer of $B$ in $\SB(H)$. 
%Fix an approximate unit $(e_\lambda)$ for $A$ included in $B$. 
%If $c\in \SM(B)$ and $a\in A$, then $ca=\lim_\lambda c e_\lambda a$. Since $c e_\lambda\in B$ for all $\lambda$, 
%$ca$ is a limit of a Cauchy net in $A$ and therefore in $A$. 
%Similarly $ac\in A$, and since $c\in \SM(B)$ was arbitrary we have $\SM(B)\subseteq \SM(A)$. 
%Since $\SM(B)\cap A=B$,  $\SM(B)/B$ is isomorphic to a subalgebra of $\SM(A)/A$. 
%\end{proof} 

From Theorem~\ref{T:mainintro} and Lemma~\ref{L.MB} we immediately have the following. 

\begin{corollary} Let $A$ be a unital separable $\Cstar$-algebra. If  $\SQ(H)$ is $2^{\aleph_0}$-universal, then so is the corona of $A\otimes\SK(H)$. \qed
\end{corollary}

We record an easy consequence of a trick first used in \cite[Theorem~4.3.11]{Phil:Class}.

\begin{proposition} \label{P:JP}
If $\kappa<2^{\aleph_0}$, there is no $\kappa$-universal \cstar-algebra. 
\end{proposition} 

\begin{proof}This follows from \cite[Theorem~2.3 and Remark 2.10]{JungePisier}.
The space $\OS_3$ of three-dimensional operator spaces can be equipped by a metric  $\delta_{\cb}$ such that the space of all operator spaces that embed into a \cstar-algebra $A$ has density at most equal to the density character of $A$ (\cite[Proposition~2.6(a)]{JungePisier}), and the space $(\OS_3,\delta_{\cb})$ has density character $2^{\aleph_0}$.\footnote{This is analogous to the fact that the space $D(T)$ of quantifier-free types in models of theory $T$ has density $2^{\aleph_0}$ whenever it is nonseparable; see \cite{kojman1992nonexistence}.}  
%Although  \cite[Theorem~2.3]{JungePisier} asserts only that the latter space is nonseparable, 
%reduces the question of the existence of such family to the existence of a family of subsets $\cF$ of $\bbN$ of cardinality $2^{\aleph_0}$ 
%such that $S\not\subseteq T$ for all distinct $S$ and $T$ in $\cF$. Such a family can be easily found.  
%associate $M^\Omega\in \OS_n$ to every $\Omega\subseteq \bbN$ so that $\Omega\not\subset\Omega'$ implies $d_{\cb}(M^\Omega, M^{\Omega'})>\e$ for a fixed $\e>0$, not depending on $\Omega$ and $\Omega'$.  To\marginpar{Remove this?} see that this implies the existence of a subset of $\OS_n$ of cardinality $2^{\aleph_0}$ any two of whose elements are at distance $>\e$, we only need to exhibit a family $\cF$ of subsets of $\bbN$ of cardinality $2^{\aleph_0}$ such that $\Omega\not\subseteq \Omega'$ for any two distinct elements $\Omega$ and $\Omega'$ of $\cF$. For this we can e.g. identify $\bbN$ with $\bbQ$ and for every irrational $r$ choose $\Omega_r$ to be a sequence in~$\bbQ$ converging to $r$. 
\end{proof}

%%%%%%%%%%

\section{The proof of Theorem~\ref{T:nuclear}}\label{S.nuclear}

A \cstar-algebra is \emph{approximately matricial}, or AM,  if it is a unital inductive limit of full matrix algebras
$M_n(\bbC)$ for $n\in \bbN$. All AM algebras are nuclear, simple, and have a unique trace 
(\cite{FaKa:Nonseparable}) and every separable AM algebra is UHF. However, nonseparable
AM algebras are not necessarily UHF, and can be quite pathological (\cite{farah2016simple}). 
Our proof of Theorem~\ref{T:nuclear} will use a class of AM algebras associated with graphs
introduced in \cite{Fa:Graphs}. 

Let $\bbP$ denote the poset for adding $\aleph_2$ Cohen reals (denoted Fn$(\aleph_2, 2, \aleph_0)$ in \cite{Ku:Set}). 
If $G\subseteq\bbP$ is generic over a model $M$ of a sufficiently large fragment of $\ZFC$ then 
every \cstar-algebra $A$ in $M$ is identified with its completion in~$M[G]$.  

\begin{lemma} \label{L.AG} Suppose that $M$ is a transitive model of a sufficiently large fragment of $\ZFC$  
and that $G$ is generic over $M$ for $\bbP$. Then there exists a \cstar-algebra $\BBB(G)$ in $M[G]$
with the following properties. 
\begin{enumerate}
\item The algebra $\BBB(G)$ is unital, nuclear, and simple.
\item Its density character is $\aleph_1$. 
\item It is not isomorphic to a subalgebra of any \cstar-algebra in $M$. 
\end{enumerate}
In addition, we can choose $\BBB(G)$ to be stably finite with a unique trace 
and faithfully representable on a separable Hilbert space.  
\end{lemma} 

The reader will notice that 
Theorem~\ref{T:mainintro}
implies that 
the \cstar-algebra $\BBB(G)$ provided by Lemma~\ref{L.AG} embeds into the Calkin algebra, 
which apparently leads to a contradiction. A reassurance that we do not have a proof that $\ZFC$ 
is inconsistent may therefore be appreciated.   
The Calkin algebra as computed in $M[G]$ 
is much larger than the completion of the  Calkin algebra as computed in $M$, 
and while $\BBB(G)$ embeds into the former by Theorem~\ref{T:mainintro} it does
not embed into the latter by the following proof. 

\begin{proof}[Proof of Lemma~\ref{L.AG}]
This proof is based on \cite[Fact on p. 889]{kojman1992nonexistence} and a construction from 
\cite{Fa:Graphs}. 
We define a simple bipartite graph $\Gamma(G)$ (isomorphic to the graph 
denoted by $G$ and/or $G^*$ in \cite{kojman1992nonexistence}) 
as follows. 
The vertex set of $\Gamma(G)$ is  $\bbN\sqcup \aleph_1$. 
Let $c_\xi$, for $\xi<\aleph_1$, enumerate the Cohen generic reals coded by $G$. 
They are functions from $\bbN$ to $\{0,1\}$, and a pair  $\{m,\xi\}$ forms an edge
if and only if 
\[
c_\xi(m)=1.
\]
Let $\BBB(G)$ be the graph CCR  \cstar-algebra defined in \cite[\S 1]{Fa:Graphs}. 
It is the universal \cstar-algebra given by the generators $u_m$, for $m\in \bbN$, 
and $v_\xi$, for $\xi<\aleph_1$ and the following relations for all $m$ and $n$ in $\bbN$ 
and all $\xi$ and $\eta$ in $\aleph_1$. 
\begin{enumerate}
\item $u_m=u_m^*$, $u_m^2=u_m^* u_m=u_m u_m^*=1$, 
\item $v_\xi=v_\xi^*$, $v_\xi^2=v_\xi^* v_\xi=v_\xi v_\xi^*=1$, 
\item $u_m u_n=u_n u_m$, 
\item $v_\xi v_\eta = v_\eta v_\xi$, 
\item $u_m v_\xi=v_\xi u_m$, if $m$ is not adjacent to $\xi$, and
\item $u_m v_\xi=-v_\xi u_m$, if $m$ is adjacent to $\xi$. 
\end{enumerate}
The set $\cP(\bbN)$ is considered with the Cantor set topology 
and the compatible metric $d(a,b)=(\min(a\Delta b)+1)^{-1}$, 

\begin{claim} For $\xi<\aleph_1$ let $a_\xi=\{m\in \bbN: c_\xi(m)=1\}$. 
The family $\bbA=\{a_\xi: \xi<\aleph_1\}$ is dense in $\cP(\bbN)$. 
Also, if $K$ and $L$ are disjoint finite subsets of $\aleph_1$ and 
$K$ is nonempty, the set 
\[
\bigcap_{\xi\in K} a_\xi\setminus \bigcup_{\eta\in L} a_\eta
\]
is infinite. (A family with this property is called \emph{independent}.) 
\end{claim} 

\begin{proof} Both properties are easy consequences of the fact that 
$\langle c_{\xi+m} : m<\omega\rangle$ does not belong to any closed nowhere dense
subset of $\cP(\bbN)^{\bbN}$ coded in $M[\langle c_\eta: \eta<\xi\rangle]$, 
for all $\xi<\aleph_1$. 
\end{proof}

 By the claim, the family $\bbA$ is dense and independent. 
 By \cite[Lemma~1.4 and the proof of Lemma~1.6]{Fa:Graphs}, the \cstar-algebra $\BBB(G)$
 is AM and it has a faithful representation on a separable Hilbert space. 
 
 It remains to prove that $\BBB(G)$ is not isomorphic to a subalgebra of any \cstar-algebra
 $A$ in $M$. Suppose otherwise, and let $\Phi\colon \BBB(G)\to A$ be a unital (and therefore necessarily injective) 
 *-homomorphism. 
 
 For $\xi<\aleph_1$ let $G_\xi$ denote the intersection of $G$ with the poset for adding the first $\xi$ 
 Cohen reals. Therefore $M[G_\xi]=M[\langle c_\eta: \eta<\xi\rangle]$. We denote this model by $M_\xi$. 
 By the ccc-ness of $\bbP$, the values of $\Phi(u_m)$ for $m\in \bbN$
 are decided by countable maximal antichains and therefore there exists $\zeta<\aleph_1$
 such that in $M_\zeta$ for every $m\in \bbN$ 
 there is $w_m\in  A$\footnote{Note that $A$ stands for the metric completion of $A$
 as computed in $M_\zeta$.} such that $\Phi(u_m)=w_m$. 

The model $M[G]$ is a forcing extension of $M_\zeta$ by the quotient of $\bbP$ with 
its regular subordering  that added $\langle c_\xi: \xi<\zeta\rangle$. This quotient ordering $\bbP^\zeta$
is isomorphic to the poset for adding $\aleph_1$ Cohen reals, and every $c_\eta$, for $\eta\geq \zeta$, 
is Cohen over $M_\zeta$ (\cite{Ku:Set}). 
 Fix $\eta\geq \zeta$ and find a condition $p$ in $G$\footnote{Formally, $p$ is in the quotient $G/G_\zeta$
 but since in the case of Cohen reals the iteration and the product coincide we can think of $p$ being in $G$.}  
 and $w\in A$ such that 
 \[
 p\forces \|w-\Phi(v_\eta)\|<1/4. 
 \]
 Fix $m\in \bbN$. Then in $M[G]$ we have (writing $[x,y]=xy-yx$)
 \[
 \|[w_m,w]-[w_m,\Phi(v_\eta)] \|\leq 2\|w_m\|\|w-\Phi(v_\eta)\|<1/2. 
 \]
 Since $w_m=\Phi(u_m)$ and $\Phi$ is an injective *-homomorphism, 
 we have $[w_m,\Phi(v_\eta)]=0$ if $c_\eta(m)=0$ and 
 $[w_m,\Phi(v_\eta)]=2w_m \Phi(v_\eta)$ if $c_\eta(m)=1$. 
 Therefore $c_\eta$ is decided by $p$ and it belongs to $M_\zeta$; 
 contradiction.  
 \end{proof}

\begin{proof}[Proof of Theorem~\ref {T:nuclear}]
To prove that it is relatively consistent with $\ZFC$  that there is no $\aleph_1$-universal nuclear
\cstar-algebra, and that there is no $\aleph_1$-universal nuclear, simple, \cstar-algebra, 
we start from a model of the Continuum Hypothesis and add $\aleph_2$ Cohen reals. 
Suppose that $A$ is a \cstar-algebra of density~$\aleph_1$ universal for 
nuclear, simple, \cstar-algebras of density character $\aleph_1$. 
 
By the ccc-ness of this poset, every \cstar-algebra of density character $\aleph_1$ is added by 
a poset for adding $\aleph_1$ Cohen reals, and by Lemma~\ref{L.AG} any further 
batch of $\aleph_1$ Cohen reals adds a unital, nuclear, and simple \cstar-algebra of density character 
$\aleph_1$ that is not isomorphic to a subalgebra of $A$; contradiction. 
\end{proof} 

We can  do  better; here is a sample (note that the \cstar-algebra $A$ is not even assumed to be nuclear, 
 and see also Remark~\ref{R.41} below).

\begin{theorem}\label{T.nuclear+}  It is relatively consistent with $\ZFC$ that 
no \cstar-algebra $A$ of density character $\aleph_{2017}$ is universal for all
unital, simple, nuclear, and stably finite   \cstar-algebras that have density character $\aleph_1$ and 
 a faithful representation 
on a separable Hilbert space. 
\end{theorem} 

\begin{proof} Start from a model of the Continuum Hypothesis and add $\aleph_{2018}$ Cohen reals. 
By the ccc-ness of the forcing, if $A$ is any \cstar-algebra of density character $<\aleph_{2018}$, then 
it belongs to an intermediate model obtained by adding $\aleph_{2017}$ Cohen reals. 
By Lemma~\ref{L.AG} any further 
batch of $\aleph_1$ Cohen reals adds 
an AM \cstar-algebra of density character 
$\aleph_1$ that has a faithful representation on a separable Hilbert space
but is not isomorphic to a subalgebra of $A$. Since AM algebras are nuclear, unital, simple, and stably finite,
this concludes the proof.   
\end{proof} 

\begin{remark} \label{R.41} The proof of Lemma~\ref{L.AG} shows that analogous statements
can be proved for models of a first-order theory (possibly in the logic of metric structures) 
with the Independence Property. This was known to the authors of \cite{kojman1992nonexistence}. 
The novelty  is that the nuclear \cstar-algebras do not form an axiomatizable class (see \cite{Muenster}). 
An observant reader familiar with model theory will have noticed that the proof of Lemma~\ref{L.AG} 
relies on the fact that the category  of nuclear, simple, and unital \cstar-algebras
includes a class of  EM-models generated by the indiscernibles. 
This fact was first used in \cite{FaKa:NonseparableII} where it was proved that there are 
$2^{\aleph_1}$ nonisomorphic AM algebras of density character~$\aleph_1$. 
Note however that the methods of \cite{FaKa:NonseparableII} do not seem to imply 
Theorem~\ref{T:nuclear}. 
\end{remark}

\section{More on the absence of universal nuclear \cstar-algebras}
\label{S:more}

We want to use the Strict Order Property (Definition~\ref{d:SOP}) to prove the following:

 \begin{theorem} \label{T.nuclear} Suppose $\kappa$ is a cardinal far from the GCH. 
  Then the following hold.  
\begin{enumerate}
\item\label{I.nuclear.1}   There exists no universal abelian \cstar-algebra of density character $\kappa$. 
\item\label{I.nuclear.2}   There exists no universal nuclear \cstar-algebra of density character $\kappa$. 
\item\label{I.nuclear.3}   There exists no universal nuclear, simple,  \cstar-algebra of density character $\kappa$. 
 \end{enumerate}
 \end{theorem}

In  case when  $\kappa=\kappa^{\aleph_0}$
 we can follow the same strategy as in Remark~\ref{R.I3}. In case this equality fails, the proof of the theorem uses \cite[Theorem~2.12]{shelah1996toward} in a manner  similar to, but easier than, 
  that in the   proof of \cite[Corollary~2.6]{shelah2006banach}. 
  While the theory of Banach spaces only has a technical weakening of  $\SOP$ known as $\SOP_4$, 
  the theory of any infinite-dimensional \cstar-algebra has the full $\SOP$ (the relation used in the 
  proof of~\cite[Lemma~5.3]{FaHaSh:Model1} is clearly transitive).

Lemma~\ref{L.nuclear} below is a consequence of \cite[Theorem~3.1]{kumjian2004certain} and  standard 
results (Remark~\ref{R.Kumjian}), but we provide a self-contained proof.  
It uses the analysis of Cuntz--Pimsner algebras associated to 
$A-B$ \cstar-correspondences and we recall the definition  (see \cite[\S 4.6]{BrownOzawa.FinDim} for more details). 
 If $A$ and $B$ are \cstar-algebras, an \emph{$A-B$ \cstar-correspondence} is a $B$-Hilbert module $E$ together  with a faithful *-representation of $A$ in the algebra of adjointable linear operators on $E$. 
 Additional properties of \cstar-correspondences used in the proof will be introduced as needed. 
 
  \begin{lemma} \label{L.nuclear} Every nuclear \cstar-algebra is isomorphic to a subalgebra of a simple, nuclear \cstar-algebra of the same density character. 
 \end{lemma}

 \begin{proof}
 Let $A$ be a \cstar-algebra. Let $\pi \colon A \to \SB(H)$ be a faithful representation of $A$ on a Hilbert space $H$ such that $\pi(A) \cap \SK(H) = \{0\}$. Such an $H$ can always be chosen to have the same density character as $A$. With $\pi$ defining a left action, we can view $H$ as an $A-\ce$ \cstar-correspondence. Consider $A$ as a Hilbert module over itself, viewed as a $\ce-A$ \cstar-correspondence (with $\ce$ acting on the left as scalars). Let $E = H \otimes_{\ce} A$.
 
 Let $\mathcal{O}_E$ be the Cuntz-Pimsner algebra associated to $E$. Then $A$ is embedded in $\mathcal{O}_E$, by construction $\mathcal{O}_E$ has the same density character, and by \cite[Theorem 4.6.25]{BrownOzawa.FinDim} and its corollary, if $A$ is nuclear then so is $\mathcal{O}_E$. (Likewise, if $A$ is exact then so is  $\mathcal{O}_E$.). 
 
 It remains to show that $\mathcal{O}_E$ is simple, which we do by verifying the conditions of \cite[Theorem 3.9]{schweizer2001dilations}. For that, we need to show that the \cstar-correspondence $E$ is full, nonperiodic and minimal. 
 
 To say that $E$ is full means that $\left < E,E \right >$ is dense in $A$. Indeed, let $a \in A$ be any positive element, and choose a unit vector $\xi \in E$, then $\left < \xi \otimes \sqrt{a}, \xi \otimes \sqrt{a} \right > = a$.  As all positive elements can be obtained, we have $\left < E,E \right > = A$.
 
 To say that $E$ is nonperiodic means that no tensor power (over $A$) of $E$, $E^{\otimes n}$, is unitarily equivalent to the trivial \cstar-correspondence $A$. By our construction, the left action of $A$ on $E^{\otimes n}$ has trivial intersection with the compacts for any $n>0$, whereas the left action on the trivial \cstar-correspondence is via compact operators, and thus they are not unitarily equivalent.
 
 To say that $E$ is minimal means that there is no nontrivial ideal $J$ in $A$ such that $\left < E,JE \right > \subseteq J$. Indeed, suppose $J$ is a nontrivial ideal. Let $a \in J$ be a non-zero element. As $\pi$ is faithful, we can choose vectors $\xi,\eta \in H$ such that $\left < \xi , \pi(a)\eta \right > = 1$. If $b \in A \smallsetminus J$ is positive  then $b = \left <\xi \otimes \sqrt{b}, a \cdot \eta \otimes \sqrt{b} \right >$ belongs to $ \left < E,JE \right >$. Thus $E$ is minimal, as required.
\end{proof}

\begin{remark} \label{R.Kumjian} Lemma~\ref{L.nuclear} is also a consequence of 
the results of \cite[Theorem~3.1]{kumjian2004certain}. 
In it Kumjian proved that every separable nuclear unital \cstar-algebra $A$ is isomorphic to a 
unital subalgebra of a separable nuclear unital simple \cstar-algebra $\cO_E$  (the UCT is not needed for this; 
see the second sentence of the proof). The separability assumption on $A$ can be removed as follows. 
If $A$ is nonseparable, then $A$ can be written as an inductive limit of a $\sigma$-closed system of 
its separable elementary submodels 
$A_\lambda$, for $\lambda\in \Lambda$. 
The Cuntz--Pimsner algebra $\cO_{E_\lambda}$ associated to $A_\lambda$ as  in \cite[\S 1]{kumjian2004certain}
is nuclear, simple, purely infinite, and unital and $\cO_E$ is the inductive limit of the system 
$\cO_{E_\lambda}$, for $\lambda\in \Lambda$. 
In addition, this system is \emph{$\sigma$-directed complete} in the sense of \cite{FaKa:Nonseparable}:  if $A_\lambda=\overline{\bigcup_n A_{\lambda_n}}$ 
then $\cO_{E_\lambda}=\overline{\bigcup_n \cO_{E_{\lambda_n}}}$. 
By a closing up argument, this implies that~$\cO_E$ has a separable elementary submodel isomorphic to $\cO_{E_\lambda}$ for some $\lambda$. 
Since being simple and purely infinite is axiomatizable (\cite[Theorem~2.15]{Muenster}), 
the conclusion follows. 
\end{remark} 

The following lemma is proved by  mimicking the proof of  \cite[Theorem~2.12]{shelah1996toward}; 
the details are worked out  in \cite{Usvyatsov-X}.

 \begin{lemma} \label{L.universal} Suppose that $\bfT$ is a theory in a continuous language 
  with the $\SOP$ and  
 that   $\kappa$ is an infinite cardinal far from the GCH. 
 Then $\bfT$ has no universal model of density character $\kappa$. \qed
\end{lemma}

\begin{remark}
	\begin{enumerate}
		\item We note that the results \cite{schweizer2001dilations} depend on results from an unpublished manuscript (number 15 in the list of references of  \cite{schweizer2001dilations}). However, a more general result was shown in \cite[Theorem~9.15]{kwasniewski2016aperiodicity}, and thus there is no gap in the literature.
		\item As the left action in our construction has trivial intersection with the compact adjointable operators, it follows that the Cuntz-Pimsner algebra $\mathcal{O}_E$ coincides with the Toeplitz-Pimsner algebra $\mathcal{T}_E$. Thus, in the separable setting, it follows from the results in Section 4 of \cite{Pimsner} that $\mathcal{O}_E$ is KK-equivalent to $A$. In the non-separable setting one cannot talk about KK-equivalence, however it follows from the results in Section 8 of \cite{Katsura} that $K_*(\mathcal{O}_E) \cong K_*(A)$ (as unordered groups). 
		\item Lemma~\ref{L.nuclear} implies that  
			the existence of a $\kappa$-universal nuclear \cstar-algebra is equivalent to the existence of a $\kappa$-universal simple nuclear \cstar-algebra for every infinite cardinal  $\kappa$.
	\end{enumerate}
\end{remark}

 \begin{proof}[Proof of Theorem~\ref{T.nuclear}] \eqref{I.nuclear.1} and \eqref{I.nuclear.2} are immediate consequences of the fact that each infinite-dimensional \cstar-algebra has SOP and Lemma~\ref{L.universal}. 
 
 \eqref{I.nuclear.3}  follows from Lemma~\ref{L.nuclear}. 
 \end{proof}

%\begin{proof}[Proof of Theorem~\ref{T:nuclear}]
%As in the proof of 
 %Corollary~\ref{C1} \eqref{I3}, by standard forcing techniques we can obtain the assertion that $\aleph_2$ is far from the GCH; then apply Theorem~\ref{T.nuclear}. 
 %\end{proof} 

% 
%Although CH  implies that there exists $\aleph_1$-universal linear ordering, 
%and the existence of an $\aleph_1$-universal linear ordering is even 
% relatively consistent with the negation of CH
%(\cite{shelah1980independence}), it is not clear whether the existence of an $\aleph_1$-universal nuclear \cstar-algebra is relatively consistent with $\ZFC$ . 
% 

\section{Concluding Remarks}\label{s:settheory}

A positive answer to the following question would imply that the 
conclusion of Theorem~\ref{T:nuclear} is independent from $\ZFC$.

\begin{question} \label{Q1} Is it relatively consistent with $\ZFC$  that there exists a universal nuclear \cstar-algebra of density character $\aleph_1$? What about a universal exact \cstar-algebra of density character $\aleph_1$? 
\end{question} 

While every separable exact \cstar-algebra is nuclearly embeddable by \cite{KirchPhil}, 
it is not known whether every exact \cstar-algebra is nuclearly embeddable.
It is therefore  not impossible that two parts of Question~\ref{Q1} have different answers. 
Theorem~\ref{T:mainintro} does not answer Question~\ref{Q1} because the Calkin algebra is neither nuclear nor exact. 

In the standard model-theoretic terminology, the following question asks whether the category of separable nuclear \cstar-algebras
has amalgamation. This is not to be confused with any of the standard amalgamation constructions used in operator algebras. 

\begin{question} \label{Q.amalg} 
 Suppose $\Phi\colon A\to B$, $\Psi\colon A\to C$ are injective $^*$-homomorphisms between nuclear \cstar-algebras. 
Is there a nuclear \cstar-algebra $D$ and injective $^*$-homomorphisms $\Phi_1\colon B\to D$, $\Psi_1\colon C\to D$ such that the diagram commutes? 
\end{question} 

We discuss briefly the connection with amalgamation in the \cstar -algebraic context. The following
idea comes from Jamie Gabe. Viewing $A$ as included in $B$ and in $C$ as above, we can form the amalgamated free product $B *_A C$.   The amalgamated free product can in general fail to be exact, let alone nuclear. If however we know that there exist conditional expectations from $B$ onto $\Phi(A)$ and from $C$ onto $\Psi(A)$, then one can form the \emph{reduced} amalgamated free product (see \cite[Section 4.7]{BrownOzawa.FinDim} for a definition and discussion). Assuming that those conditional expectations have faithful GNS representations (which is automatic if, for example, all \cstar -algebras in questions are simple), it follows from \cite[Corollary 5.7]{Dykema-Shlyakhtenko} that the reduced amalgamated free product of $B$ and $C$ over $A$ is exact. By Kirchberg's embedding theorem, the reduced amalgamated free product then embeds in $\mathcal{O}_2$. This gives a partial positive answer to Question \ref{Q.amalg}.

By using standard techniques, a positive answer to Question~\ref{Q.amalg}  would imply that under the Continuum Hypothesis there exists an $\aleph_1$-universal nuclear \cstar-algebra.
The following lemma shows that it suffices to answer  an `easier' version of  Question~\ref{Q.amalg}. 

\begin{lemma} \label{L.O2} 
Every nuclear \cstar-algebra $A$ is isomorphic to a subalgebra of a \cstar-algebra that 
is an inductive limit of a net of  \cstar-algebras  each of which is isomorphic to $\cO_2$. 
\end{lemma} 

\begin{proof} By  Lemma~\ref{L.nuclear}, 
 $A$ is isomorphic to a subalgebra of a  simple, nuclear \cstar-algebra $B$ of the same density character. 
We claim that  $C=B\otimes \cO_2$ is as required. Since $B$ is simple, it is equal to the inductive limit of 
simple and separable \cstar-algebras. We can moreover assure that these algebras are nuclear, either by closing up 
or by taking an elementary submodel and using \cite[Theorem~5.7.3]{Muenster}. 
But if $B=\lim_\lambda B_\lambda$, then $C=\lim_\lambda B_\lambda\otimes \cO_2$, 
and by Kirchberg's $\cO_2$-absorption theorem $B_\lambda\otimes \cO_2\cong \cO_2$ for all $\lambda$. 
\end{proof} 

In relation to the conclusion of Lemma~\ref{L.O2} it should be noted that inductive limit of \cstar-algebras isomorphic to $\cO_2$
can be quite unruly. By the main result of \cite{farah2016simple}, Jensen's $\Diamond$ (\cite[III.7]{Ku:Set}) implies the existence 
of a \cstar-algebra that is an inductive limit of \cstar-algebras isomorphic to $\cO_2$ but it is not isomorphic to its opposite algebra, 
all of its irreducible representations are unitarily equivalent, and all of its automorphisms are inner.

\begin{question} \label{Q.amalg.1} 
 Suppose $\Phi\colon \cO_2\to \cO_2$, $\Psi\colon \cO_2\to \cO_2$ are unital $^*$-homomorphisms. 
 Are there unital $^*$-homomorphisms $\Phi_1\colon \cO_2\to \cO_2$, $\Psi_1\colon \cO_2\to \cO_2$ such that $\Phi_1\circ \Phi=\Psi_1\circ\Psi$? 
\end{question} 

As in the discussion following Question \ref{Q.amalg}, if the images of $\Phi$ and $\Psi$ admit conditional expectations onto them then the answer is positive. We do not know whether there are subalgebras of $\cO_2$ isomorphic to $\cO_2$ which do not admit conditional expectations onto them. 

By using Lemma~\ref{L.O2},  standard techniques show that a positive answer to Question~\ref{Q.amalg.1} would imply a positive answer to 
Question~\ref{Q1}. A standard descriptive set-theoretic argument shows that a a positive answer to Question~\ref{Q.amalg.1} is equivalent 
to a  $\Sigma^1_2$ statement and therefore absolute between transitive models of ZFC containing all countable ordinals. Because of this
the answer to this question is unlikely to be independent from ZFC. For example, if this question can be resolved by using the Continuum Hypothesis (or $\Diamond$, or Martin's Axiom\dots) 
then it can be resolved in ZFC alone. This however still leaves a possibility that Question~\ref{Q.amalg.1} cannot be resolved in ZFC. 
 See \cite[\S 3 and \S A.4]{Fa:Absoluteness} for a discussion of the absoluteness phenomenon. 

We can also search for potential target algebras to replace $\cO_2$. A unital purely infinite and simple \cstar-algebra $A$ is in Cuntz standard form if $[1_A]=0$ in $K_0(A)$. It can be shown that $A$ is in standard form if and only if $\cO_2$ embeds unitally in $A$. Every Kirchberg algebra is stably isomorphic to one in standard form, denoted $A^{st}$. This association is unique up to isomorphism. Given this, as $\cO_2$ embeds unitally into $\cO_{\infty}^{st}$, $\cO_{\infty}^{st}$ serves as a universal algebra which admits unital embeddings of any separable nuclear \cstar -algebra. Denote by $\cO_{\aleph_1}$ the analogue of $\cO_{\infty}$ corresponding to $\aleph_1$-many isometries with pairwise orthogonal ranges. 

\begin{question}
Is  $\cO_{\aleph_1}^{st}$ a universal nuclear \cstar -algebra of density character $\aleph_1$?
\end{question}
%\begin{proposition} If Question~\ref{Q.amalg}  has a positive answer then 
% the Continuum Hypothesis implies there exists an $\aleph_1$-universal nuclear \cstar-algebra. \qed
%\end{proposition} 

\subsection{Remarks on universality in related categories}\label{s:settheory}

\subsection*{Isomorphic embeddings of Banach spaces} 
Theorem~\ref{T:mainintro} was inspired by   \cite[Theorem~1.4]{BrechKosz}, where the analogous statement for $2^{\aleph_0}$-univeral Banach spaces was proved.  Brech and Koszmider constructed a forcing extension in which an isometrically $2^{\aleph_0}$-Banach space exists, but $\ell_\infty/c_0$ is not isometrically, or even isomorphically, $2^{\aleph_0}$-universal Banach space.  The result of  \cite[Corollary~2.4]{shelah2006banach} used in the proof of Corollary~\ref{C1} was improved in \cite[Theorem~1.3]{BrechKosz}, where it was proved that consistently there is no isomorphically $2^{\aleph_0}$-universal Banach space. 

\subsection*{Linear orders} The existence of universal linear orders is a well-studied 
subject (\cite{kojman1992nonexistence}). 
Much attention has been devoted to the question of $2^{\aleph_0}$-universality of $\cP(\bbN)/\Fin$.  
Since the Calkin algebra is its noncommutative analogue  (see e.g. \cite{We:Set}), 
we shall concentrate on the role of $\cP(\bbN)/\Fin$. 
 While it is consistent that the Continuum Hypothesis fails and $\cP(\NN)/\Fin$ is $2^{\aleph_0}$-universal (\cite{laver1979linear}), it is not clear whether the assertion `$\SQ(H)$ is a $2^{\aleph_0}$-universal \cstar-algebra'  is relatively consistent with the failure of the Continuum Hypothesis. The question on
whether for a given \cstar-algebra  $A$ there exists a ccc forcing notion that forces an embedding of $A$ into $\SQ(H)$ was given a positive answer in \cite{FarKatVac:Emb}. Notably, the structure of the small category of linear orders that embed into $\cP(\bbN)/\Fin$ %or $(\bbN^{\bbN},\preceq)$
is remarkably malleable in $\ZFC$   (see \cite[\S 1]{Fa:Embedding}) and very rigid if a fragment of PFA holds (\cite{Farah.AQ}).  

\subsection*{Surjective universality for compact Hausdorff spaces}
A compact Hausdorff space $X$ is said to be $\kappa$-universal if it is surjectively universal  among compact Hausdorff spaces of weight $\kappa$. 
By Gelfand--Naimark duality, this is equivalent to $C(X)$ being an injectively universal unital abelian \cstar-algebra. The Continuum Hypothesis implies   that $\beta\bbN\setminus\bbN$ is an $\aleph_1$-universal compact Hausdorff space (Parovi\v cenko's theorem) and that  $\beta\bbR_+\setminus \bbR_+$  
is an $\aleph_1$-universal connected compact Hausdorff space (\cite{DoHa:Universal}).  As in Corollary~\ref{C1}, PFA implies that 
$\beta\bbN\setminus \bbN$ is not $2^{\aleph_0}$-universal because it does not map onto 
the Stone space of the Lebesgue measure algebra (\cite{DoHa:Lebesgue}).

\subsection*{II$_1$-factors}
In \cite{Oz:There} it was proved that there is no $\kappa$-universal II$_1$-factor for any $\kappa<2^{\aleph_0}$. 
As in Corollary~\ref{C1}, $\kappa^{<\kappa}=\kappa$ implies there is a $\kappa$-universal II$_1$-factor. 
 The theory of II$_1$-factors has the  Order Property
 (\cite[Lemma~3.2]{FaHaSh:Model1}) but it is not known whether it has the Strict Order Property. 
If it does,   the argument from 
Remark~\ref{R.I3} would  imply that the existence of a cardinal $\lambda$ such that $\lambda^+<2^{\aleph_0}<2^\lambda$
implies there is no $2^{\aleph_0}$-universal II$_1$-factor.  
As a curiosity, we note that  Connes' Embedding Problem has the positive solution if and only if 
the Continuum Hypothesis implies that an ultrapower of the hyperfinite II$_1$-factor 
is a $2^{\aleph_0}$-universal II$_1$-factor. Similarly, 
 Kirchberg's Embedding Problem 
(\cite{goldbring2014kirchberg}) has the positive solution if and only if 
the Continuum Hypothesis implies that 
an ultrapower of $\cO_2$ is a $2^{\aleph_0}$-universal \cstar-algebra.

\bibliographystyle{plain} %{amsalpha}
\bibliography{library}
\end{document}